\documentclass[12pt,twoside]{amsart}
\usepackage{amssymb}
\usepackage{amscd}
\usepackage{color}

\newenvironment{claim}[1]{\par\noindent\underline{Claim}\space#1}{}
\newenvironment{claimproof}[1]{\par\noindent\underline{Proof of Claim}\space#1}{\hfill $\blacksquare$}

\title[Frobenius push-forward of the structure sheaf]
{A characterization of ordinary abelian varieties 
by the Frobenius push-forward of the structure sheaf} 

\author{Akiyoshi Sannai and Hiromu Tanaka} 
\address{Graduate School of Mathematical Sciences, University of Tokyo,
Meguro, Tokyo, 153-9814, Japan.}
\email{sannai@ms.u-tokyo.ac.jp}
\address{Department of Mathematics, Imperial College, London, 180 Queen's Gate, 
London SW7 2AZ, UK} 
\email{h.tanaka@imperial.ac.uk}

\newcommand{\rank}[0]{{\operatorname{rank}}}

\newcommand{\Alb}[0]{{\operatorname{Alb}}}
\newcommand{\Ker}[0]{{\operatorname{Ker}}}
\newcommand{\Image}[0]{{\operatorname{Image}}}

\newcommand{\Spec}[0]{{\operatorname{Spec}}}

\newcommand{\Pic}[0]{{\operatorname{Pic}}}

\newcommand{\red}[0]{{\operatorname{red}}}
\newcommand{\Ex}[0]{{\operatorname{Ex}}}
\newtheorem{thm}{Theorem}[section]
\newtheorem{lem}[thm]{Lemma}

\newtheorem{prop}[thm]{Proposition}

\theoremstyle{definition}
\newtheorem{ques}[thm]{Question}

\newtheorem{dfn}[thm]{Definition}
\newtheorem{dfnprop}[thm]{Definition-Proposition}
\newtheorem{rem}[thm]{Remark}
\newtheorem*{ack}{Acknowledgments}      
         
\newtheorem{step}{Step}

\newcommand{\MO}{\mathcal{O}}

\begin{document}

\maketitle

\begin{abstract}
For an ordinary abelian variety $X$, 
$F^e_*\mathcal{O}_X$ is decomposed into line bundles for every positive integer $e$. 
Conversely, if a smooth projective variety $X$ satisfies this property and 
the Kodaira dimension of $X$ is non-negative, then $X$ is an ordinary abelian variety. 
\end{abstract}

\tableofcontents
\setcounter{section}{0}

\section{Introduction}

Let $k$ be an algebraically closed field of characteristic $p>0$. 
Let $X$ be a smooth proper variety over $k$. 
When does $X$ satisfy the following property $(*)$? 

\begin{enumerate}
\item[$(*)$]{$F_*\MO_X\simeq \bigoplus_j M_j$ where $F:X \to X$ is the absolute Frobenius morphism and 
each $M_j$ is a line bundle. }
\end{enumerate}

For example, an arbitrary smooth proper toric variety satisfies this property $(*)$ (cf. \cite{Achinger}\cite{Th}). 
Thus there are many varieties which satisfy $(*)$. 
But every toric variety has negative Kodaira dimension. 
On the other hand, we show that ordinary abelian varieties satisfy $(*)$. 
The main theorem of this paper is the following inverse result.

\begin{thm}[Theorem~\ref{main-theorem2}]\label{0main1}
Let $k$ be an algebraically closed field of characteristic $p>0$. 
Let $X$ be a smooth projective variety over $k$. 
Assume the following conditions. 
\begin{itemize}
\item{For infinitely many $e\in\mathbb Z_{>0}$, 
$F^e_*\MO_X\simeq \bigoplus_j M_j^{(e)}$ where each $M_j^{(e)}$ is an invertible sheaf. }
\item{$K_X$ is pseudo-effective (e.g. the Kodaira dimension of $X$ is non-negative).}
\end{itemize}
Then $X$ is an ordinary abelian variety. 
\end{thm}


On the other hand, how about the opposite problem? 
More precisely, when does $X$ satisfy the following property $(**)$? 

\begin{enumerate}
\item[$(**)$]{$F_*\MO_X$ is indecomposable, that is, 
if $F_*\MO_X=E_1\oplus E_2$ holds for some coherent sheaves $E_1$ and $E_2$, then 
$E_1=0$ or $E_2=0$. }
\end{enumerate}

We study this problem for abelian varieties and curves.

\begin{thm}[Theorem~\ref{abel-main}]\label{0abel-main}
Let $k$ be an algebraically closed field of characteristic $p>0$. 
Let $X$ be an abelian variety over $k$. 
Set $r_X$ to be the $p$-rank of $X$. 
Then, for every $e\in\mathbb Z_{>0}$, 
$$F_*^e\MO_X \simeq E_1 \oplus \cdots \oplus E_{p^{er_X}}$$
where each $E_i$ is an indecomposable locally free sheaf of rank $p^{e(\dim X-r_X)}$. 
In particular, $F_*^e\MO_X$ is indecomposable if and only if $r_X=0$. 
\end{thm}

\begin{thm}[Theorem~\ref{main-curve}]\label{0main2}
Let $k$ be an algebraically closed field of characteristic $p>0$. 
Let $X$ be a smooth projective curve of genus $g$. 
Fix an arbitrary integer $e\in\mathbb Z_{>0}$. 
Then the following assertions hold. 
\begin{enumerate}
\item[$(0)$]{If $g=0$, then $F^e_*\MO_X \simeq \bigoplus L_j$ where 
every $L_j$ is a line bundle.}
\item[$(1{\rm or})$]{If $g=1$ and $X$ is an ordinary elliptic curve, 
then $F^e_*\MO_X \simeq \bigoplus L_j$ where every $L_j$ is a line bundle.}
\item[$(1{ \rm ss})$]{If $g=1$ and $X$ is a supersingular elliptic curve, 
then $F^e_*\MO_X$ is indecomposable.}
\item[$(2)$]{If $g\geq 2$, then $F^e_*\MO_X$ is indecomposable.}
\end{enumerate}
\end{thm}

By Theorem~\ref{0main2}(2), it is natural to ask whether 
$F_*\MO_X$ is indecomposable for every smooth projective variety of general type $X$. 
If we drop the assumption that $X$ is smooth, then 
the following theorem gives a negative answer to this question.

\begin{thm}\label{0general-type}
Let $k$ be an algebraically closed field of characteristic $p>0$. 
Then, there exists a projective normal surface $X$ over $k$ which satisfies the following properties. 
\begin{enumerate}
\item{The singularities of $X$ are at worst canonical.}
\item{$K_X$ is ample.}
\item{$F_*\MO_X$ is not indecomposable.}
\end{enumerate}
\end{thm}

\begin{rem}
By \cite{Sun}, 
if $X$ is a smooth projective curve of genus $g\geq 2$, 
then  $F_*E$ is a stable vector bundle whenever so is $E$. 
Theorem~\ref{0general-type} shows that 
there exists a projective normal canonical surface of general type $X$ 
such that $F_*\MO_X$ is not a stable vector bundle with respect to 
an arbitrary ample invertible sheaf $H$ on $X$. 
\end{rem}

{\bf Proof of Theorem~\ref{0main1}:} 
We overview the proof of Theorem~\ref{0main1}. 
First of all, we can show that $X$ is $F$-split, that is, 
$\MO_X \to F_*\MO_X$ splits as an $\MO_X$-module homomorphism. 
This implies 
$$H^0(X, -(p-1)K_X)\neq 0.$$ 
Since $K_X$ is pseudo-effective, we obtain $(p-1)K_X \sim 0$. 
Then, by \cite{HP}, we see that the Albanese map $\alpha:X \to \Alb(X)$ is surjective. 
There are two main difficulties as follows. 
\begin{enumerate}
\item{To show that $\alpha$ is generically finite.}
\item{To treat the case where $\alpha$ is a finite surjective inseparable morphism.}
\end{enumerate}

(1) 
Let us overview how to show that $\alpha$ is generically finite. 
Set $r_X$ to be the $p$-rank of $\Alb(X)$. 
It suffices to show $\dim X=r_X$. 
Note that $\alpha:X \to \Alb(X)$ induces the following bijective group homomorphism: 
$$\alpha^*:\Pic^0(\Alb(X)) \overset{\simeq}\to \Pic^0(X),\,\,\, L\mapsto \alpha^*L.$$
Roughly speaking, 
since $\Pic^{\tau}(X)/\Pic^0(X)$ is a finite group, 
$r_X$ can be calculated by the asymptotic behavior of the number of $p^e$-torsion line bundles on $X$. 
Thus, we count the number of $p^e$-torsion line bundles on $X$. 
More precisely, we prove that the number of $p^e$-torsion line bundles on $X$ is $p^{e\dim X}$ 
for infinitely many $e$. 
 

Now we have 
$$F^e_*\MO_X=\bigoplus_{1\leq j\leq p^{e\dim X}} M_j$$
where each $M_j$ is a line bundle. 
In our situation, 
we can show that every $p^e$-torsion line bundle $L$ 
is isomorphic to some $M_j$ (cf. Lemma~\ref{p-torsion}). 
Therefore, it suffices to prove that each $M_j$ is $p^e$-torsion. 
Tensor $M_j^{-1}$ with the above equation and take $H^0$. 
Then, we obtain $H^0(X, M_j^{-p^e})\neq 0$. 
If we have $H^0(X, M_j^{p^e})\neq 0$, then we are done. 
For this, we take the duality, that is, 
apply $\mathcal Hom_{\MO_X}(-, \omega_X)$ to the above direct summand decomposition. 
Then we can also show $H^0(X, M_j^{p^e})\neq 0$. 
For more details on this argument, see Lemma~\ref{CY-p-torsion}. 

(2) 
We overview how to treat the inseparable case. 
To clarify the idea, we assume that 
$\alpha$ is a finite surjective purely inseparable morphism of degree $p$. 
Then, Frobenius map $F_A$ of $A$ factors through $\alpha$: 
$$F_A:A \to X \overset{\alpha}\to A.$$
By using the fact that $X$ is $F$-split, 
we can show that 
$$(F_A)_*\MO_A \simeq \alpha_*\MO_X\oplus E$$
for some coherent sheaf $E$. 
Since $(F_A)_*\MO_A$ is the direct sum of the $p$-torsion line bundles, 
we obtain 
$$\alpha_*\MO_X\simeq \bigoplus_{j=1}^p M_j$$
where $M_1, \cdots, M_p$ are mutually distinct $p$-torsion line bundles. 
One of them, say $M_1$, satisfies $\alpha^*M_1 \not\simeq \MO_X$. 
By tensoring $M_1^{-1}$, 
we obtain 
$$\alpha_*(\alpha^*M_1^{-1})\simeq \MO_A \oplus \bigoplus_{j=2}^p (M_j\otimes_{\MO_A}M_1^{-1})$$
which induces the following contradiction: 
$$0=H^0(X, \alpha^*M_1^{-1}) \simeq H^0(A, \MO_A)\oplus H^0(A, \bigoplus_{j=2}^p (M_j\otimes_{\MO_A}M_1^{-1}))\neq 0.$$
In the proof of Theorem~\ref{0main1}, there appear other technical difficulties. 
For more details on the inseparable case, see Step~\ref{p-insep} of the proof of Theorem~\ref{main-theorem2}. 


{\bf Related results:}
(1) 
In \cite{HP}, Hacon and Patakfalvi give a characterization 
of the varieties birational to ordinary abelian varieties. 

(2) \cite{Achinger} gives a characterization of smooth projective toric varieties as follows. 
For a smooth projective variety $X$ in positive characteristic, 
$X$ is toric if and only if 
$F_*L$ splits into line bundles for every line bundle $L$.

\begin{ack}
The authors would like to thank Professors 
Piotr Achinger, Yoshinori Gongyo, Nobuo Hara, Masayuki Hirokado, 
Kazuhiko Kurano, and Shunsuke Takagi, Mingshuo Zhou for several useful comments. 
We are grateful to the referee for valuable comments. 
The first author is partially supported by 
the Grant-in-Aid for JSPS Fellows (24-0745).
\end{ack}

\section{Preliminaries}\label{Preliminaries}

\subsection{Notation}\label{Notation}


We will not distinguish the notations line bundles, invertible sheaves and Cartier divisors. 
For example, we will write $L+M$ for line bundles $L$ and $M$. 

Throughout this paper, 
we work over an algebraically closed field $k$ of characteristic $p>0$. 
For example, a projective scheme means 
a scheme which is projective over $k$. 

Let $X$ be a noetherian scheme. 
For a coherent sheaf $F$ on $X$ and a line bundle $L$ on $X$, 
we define $F(L):=F\otimes_{\MO_X} L$. 

In this paper, a {\em variety} means 
an integral scheme which is separated and of finite type over $k$. 
A {\em curve} or a {\em surface} means a variety 
whose dimension is one or two, respectively. 

For a proper scheme $X$, 
let $\Pic(X)$ be the group of line bundles on $X$ and 
let $\Pic^0(X)$ (resp. $\Pic^{\tau}(X)$) be the subgroup of $\Pic(X)$ 
of line bundles which are algebraically (resp. numerically) equivalent to zero: 
$$\Pic^0(X) \subset \Pic^{\tau}(X) \subset \Pic(X).$$

For a normal variety $X$ and a coherent sheaf $M$ on $X$, 
we say $M$ is {\em reflexive} if the natural map 
$M \to \mathcal Hom_{\MO_X}(\mathcal Hom_{\MO_X}(M, \MO_X), \MO_X)$ 
is an isomorphism. 
We say $M$ is {\em divisorial} if $M$ is reflexive and 
$M|_{\MO_{X, \xi}}$ is rank one where $\xi$ is the generic point. 
It is well-known that a divisorial sheaf $M$ is isomorphic to the sheaf $\MO_X(D)$ 
associated to a Weil divisor $D$ on $X$.

Let $X$ be a scheme of finite type over $k$. 
We say $X$ is $F$-{\em split} if the absolute Frobenius 
$$\MO_X \to F_*\MO_X,\,\,\, a\mapsto a^p$$
splits as an $\MO_X$-module homomorphism. 
 
We say a coherent sheaf $F$ is {\em indecomposable} 
if, for every isomorphism $F\simeq F_1\oplus F_2$ with coherent sheaves $F_1$ and $F_2$, 
we obtain $F_1=0$ or $F_2=0$.

We recall the definition of ordinary abelian varieties. 

\begin{dfnprop}
Let $X$ be an abelian variety. 
We say $X$ is {\em ordinary} if one of the following conditions hold. 
Moreover, the following conditions are equivalent. 
\begin{enumerate}
\item{For some $e\in\mathbb Z_{>0}$, the number of $p^e$-torsion points is $p^{e\cdot\dim X}$.}
\item{For every $e\in\mathbb Z_{>0}$, the number of $p^e$-torsion points is $p^{e\cdot\dim X}$.}
\item{$F:H^1(X, \MO_X) \to H^1(X, \MO_X)$ is bijective.}
\item{$F:H^i(X, \MO_X) \to H^i(X, \MO_X)$ is bijective for every $i\geq 0$.}
\item{$X$ is $F$-split.}
\end{enumerate}
\end{dfnprop}

\begin{proof}
(1) and (2) are equivalent by \cite[Section~15, The $p$-rank]{Mumford}. 
(2) and (3) are equivalent by \cite[Section~15, Theorem~3]{Mumford}. 
(Note that, in older editions of \cite{Mumford}, there are two Theorem~2 in Section~15.) 
(3) and (4) are equivalent by \cite[Example~5.4]{MuSr}. 
(4) and (5) are equivalent by \cite[Lemma~1.1]{MeSr}. 
\end{proof}

\subsection{Albanese varieties}\label{Albanese}

In this subsection, we recall the definition and fundamental properties 
of the Albanese varieties. 
For details, see \cite[Section~9]{FGAex}. 




For a projective normal variety $X$ and a closed point $x\in X$, 
there uniquely exists a morphism $\alpha_X:X \to \Alb(X)$ 
to an abelian variety $\Alb(X)$, called the {\em Albanese variety} of $X$, 
such that $\alpha_X(x)=0$ and that 
every morphism to an abelian variety $g:X \to B$, with $g(x)=0_B$, 
factors through $\alpha_X$ (cf. \cite[Remark~9.5.25]{FGAex}). 
Note that $\Alb(X) \simeq \underline{\Pic}^0(\underline{\Pic}^0(X)_{\red})$, 
where $\underline{\Pic}(X):=\mathbf{Pic}_{X/k}$ in the sense of \cite[Section~9]{FGAex}. 

The Albanese morphism $\alpha_X:X \to \Alb(X)$ induces a natural morphism 
$$\alpha_X^*:\underline{\Pic}^0(\Alb(X)) \to \underline{\Pic}^0(X)_{\red}.$$
It is well-known that $\alpha_X^*$ is an isomorphism. 
In particular, the induced group homomorphism 
$$\alpha_X^*:\Pic^0(\Alb(X)) \to \Pic^0(X)$$ 
is bijective.

\subsection{The number of $p^e$-torsion line bundles}

The asymptotic behavior of the number of $p^e$-torsion line bundles 
is determined by the $p$-rank of the Picard variety $\Pic^0(X)_{\red}$. 

\begin{prop}\label{picard-exact}
Let $X$ be a projective normal variety. 
Then, the following assertions hold. 
\begin{enumerate}
\item{There exists the following exact sequence 
$$0 \to \Pic^0(X) \to \Pic^{\tau}(X) \to G(X) \to 0$$
where $G(X)$ is a finite group. }
\item{If $r_X$ is the $p$-rank of $\underline{\Pic}^0(X)_{\red}$, 
then there exists $\xi \in\mathbb Z_{>0}$ such that 
$$p^{er_X} \leq |\Pic(X)[p^e]| \leq p^{er_X}\times \xi$$
for every $e\in\mathbb Z_{>0}$ where 
$\Pic(X)[p^e]$ is the group of $p^e$-torsion line bundles. }
\end{enumerate}
\end{prop}

\begin{proof}
The assertion (1) holds by \cite[9.6.17]{FGAex}. 
The assertion (2) follows from (1). 
\end{proof}

As a consequence, we see that the $p$-rank of the Picard variety 
is stable under purely inseparable covers.  

\begin{prop}\label{p-rank-insep}
Let $f: X \to Y$ be a finite surjective purely inseparable morphism between 
projective normal varieties. 
Set $r_X$ and $r_Y$ to be the $p$-ranks of $\underline{\Pic}^0(X)_{\red}$ and $\underline{\Pic}^0(Y)_{\red}$, respectively. 
Then, $r_X=r_Y$. 
\end{prop}

\begin{proof}
We may assume that $[K(X):K(Y)]=p$. 
Then, the absolute Frobenius morphism $F:Y \to Y$ factors through $f:X \to Y$: 
$$F:Y \overset{g}\to X \overset{f}\to Y.$$
Thus, it suffices to show $r_Y \leq r_X$. 

We show that the following inequality 
$$p^{er_Y} \leq |\Pic(X)[p^{e+1}]|$$
holds for every $e\in\mathbb Z_{>0}$. 
Fix $e\in\mathbb Z_{>0}$. 
Let $L_1, \cdots, L_{p^{er_Y}}$ be mutually distinct $p^e$-torsion line bundles in $\Pic^0(Y)$. 
Then, since $\underline{\Pic}^0(Y)_{{\red}}$ is an abelian variety, 
we can find line bundles $M_1, \cdots, M_{p^{er_Y}}$ such that 
$M_j^p\simeq L_j$ for every $1\leq j \leq p^{er_Y}$. 
We see that, for each $j$, 
$$L_j \simeq M_j^p =F^*M_j\simeq g^*f^*M_j$$
and that $f^*M_1, \cdots, f^*M_{p^{er_Y}}$ are mutually distinct $p^{e+1}$-torsion line bundles on $X$. 
Thus, we obtain the required inequality $p^{er_Y}\leq |\Pic(X)[p^{e+1}]|.$

By Proposition~\ref{picard-exact}(2), we can find $\xi \in\mathbb Z_{>0}$ such that 
the inequalities 
$$p^{er_Y} \leq |\Pic(X)[p^{e+1}]| \leq p^{(e+1)r_X}\times \xi$$
hold for every $e\in\mathbb Z_{>0}$. 
By taking the limit $e\to \infty$, we obtain $r_Y \leq r_X$. 
\end{proof}

\section{Basic properties}\label{Basic}

In the main theorem (Theorem~\ref{0main1}), 
we treat varieties 
such that $F_*^e\MO_X$ is decomposed into line bundles. 
In this section, we summarize basic properties of such varieties. 
Since such varieties are $F$-split (Lemma~\ref{F-split}), 
we also study $F$-split varieties. 
First, we give characterizations of $F$-split varieties.

\begin{lem}\label{F-split-lemma}
Let $X$ be a scheme of finite type over $k$. 
Then, the following assertions are equivalent. 
\begin{enumerate}
\item{$X$ is $F$-split.}
\item{For every $e \in \mathbb Z_{>0}$, 
there exists a coherent sheaf $E$ such that 
$F_*^e\MO_X\simeq \MO_X\oplus E$.}
\item{$F_*^e\MO_X\simeq \MO_X\oplus E$ for some $e\in\mathbb Z_{>0}$ and coherent sheaf $E$.}
\item{$F_*^e\MO_X\simeq L\oplus E$ for some $e\in\mathbb Z_{>0}$, $p^e$-torsion line bundle $L$ 
and coherent sheaf $E$.}
\end{enumerate}
\end{lem}

\begin{proof}
It is well-known that (1), (2) and (3) are equivalent. 
It is clear that (3) implies (4). 
We see that (4) implies (3) by tensoring $L^{-1}$ with $F_*^e\MO_X\simeq L\oplus E$. 

\end{proof}

We are interested in varieties such that $F_*^e\MO_X$ is decomposed into line bundles. 
By the following lemma, such varieties are $F$-split.

\begin{lem}\label{F-split}
Let $X$ be a proper normal variety. 
Assume that $F_*^e\MO_X\simeq \bigoplus_j M_j$ for some $e\in\mathbb Z_{>0}$ and 
divisorial sheaves $M_j$. 
Then, $X$ is $F$-split.
\end{lem}

\begin{proof}
We obtain the following 
$$0\neq H^0(X, F_*^e\MO_X)\simeq \bigoplus_j H^0(X, M_j).$$
Therefore, we see 
$H^0(X, M_{j_0})\neq 0$ for some $j_0$. 
We have $M_{j_0}\simeq \MO_X(E)$ for some effective divisor $E$ on $X$. 
By Lemma~\ref{F-split-lemma}, 
it is enough to show $E=0$. 
Tensor $\MO_X(-E)$ with 
$$F_*^e\MO_X\simeq \bigoplus_j M_j\simeq \MO_X(E)\oplus(\bigoplus_{j\neq j_0} M_j)$$
and take the double dual. 
We obtain the following decomposition: 
$$F_*^e(\MO_X(-p^eE))\simeq \MO_X\oplus(\bigoplus_{j\neq j_0} (M_j\otimes_{\MO_X} \MO_X(-E))^{**}).$$
Thus, $H^0(X, \MO_X(-p^eE))\neq 0$. 
This implies $E=0$. 
\end{proof}

The following result gives an upper bound of the number of 
$p^e$-torsion line bundles for $F$-split varieties. 

\begin{lem}\label{p-torsion}
Let $X$ be a proper variety. 
Assume that $X$ is $F$-split. 
Fix $e\in\mathbb Z_{>0}$. 
Let $F_*^e\MO_X\simeq \bigoplus_{j\in J} M_j$ be a decomposition 
into indecomposable coherent sheaves $M_j$. 
Then, the following assertions hold. 
\begin{enumerate}
\item{Let $L$ be a line bundle with $L^{p^e}\simeq \MO_X$. 
Then, $L\simeq M_{j_1}$ for some $j_1\in J$.}
\item{Let $j_1, j_2\in J$ with $j_1\neq j_2$. 
If $M_{j_1}$ and $M_{j_2}$ are line bundles and 
satisfy $M_{j_1}^{p^e}\simeq \MO_X$ and $M_{j_2}^{p^e}\simeq \MO_X$, 
then $M_{j_1}\not\simeq M_{j_2}$.}
\item{The number of $p^e$-torsion line bundles on $X$ is 
at most $p^{e\cdot \dim X}$.}
\end{enumerate}
\end{lem}

\begin{proof}
(1) 
Tensor $L^{-1}$ with $F_*^e\MO_X\simeq \bigoplus_j M_j$ and we obtain 
$$F_*^e\MO_X\simeq F_*^e(L^{-p^e})\simeq F_*^e\MO_X\otimes_{\MO_X} L^{-1}\simeq \bigoplus_j (M_j\otimes_{\MO_X} L^{-1}).$$
Since $X$ is $F$-split, we have 
$$F_*^e\MO_X\simeq \MO_X\oplus (\bigoplus_i N_i)$$
where each $N_i$ is an indecomposable sheaf. 
Then, the Krull--Schmidt theorem (\cite[Theorem~2]{Atiyah}) implies $M_{j_1}\otimes_{\MO_X} L^{-1}\simeq \MO_X$ for some $j_1$. 

(2) 
Assume that, for some $j_1\neq j_2$, $M_{j_1}$ and $M_{j_2}$ are line bundles 
such that $M_{j_1}^{p^e}\simeq \MO_X$, $M_{j_2}^{p^e}\simeq \MO_X$ and $M_{j_1}\simeq M_{j_2}$. 
Let us derive a contradiction. 
Tensor $M_{j_1}^{-1}$ and we obtain 
$$F_*^e\MO_X\simeq F_*^e(M_{j_1}^{-p^e})\simeq \MO_X\oplus \MO_X\oplus(\bigoplus_{j\neq j_1, j_2} (M_j\otimes M_{j_1}^{-1})).$$
Taking $H^0$, we obtain a contradiction. 

(3) 
The assertion immediately follows from (1) and (2). 
\end{proof}

The following lemma is used in the next section and 
well-known for experts on $F$-singularities (cf. the proof of \cite[Theorem~4.3]{SS}). 

\begin{lem}\label{Kodaira-negative}
Let $X$ be a smooth proper variety. 
Assume that $X$ is $F$-split. 
Then, for every $e\in\mathbb Z_{>0}$, 
$$H^0(X, -(p^e-1)K_X)\neq 0.$$ 
In particular, $\kappa(X)\leq 0.$
\end{lem}

\begin{proof}
By the Grothendieck duality, we can check 
$$\mathcal Hom_{\MO_X}(F_*^e\MO_X, \omega_X)\simeq F_*^e\omega_X.$$ 
This implies that $\omega_X$ is a direct summand of $F_*^e\omega_X$, 
which is equivalent to the assertion that $\MO_X$ is a direct summand of $F_*^e(\omega_X^{1-p^e})$. 
\end{proof}


\section{A characterization of ordinary abelian varieties}\label{section-main}






In this section, we show the main theorem of this paper: Theorem~\ref{main-theorem2}. 
In the proof, we use \cite[Theorem~1.1.1]{HP}. 
For this, it is necessary to show $\kappa_S(X)=0$. 
We check this in Lemma~\ref{K-torsion}. 
First, we recall the definition of $\kappa_S(X)$. 

\begin{dfn}\label{def-S0}
Let $X$ be a smooth proper variety. 
\begin{enumerate}
\item{Fix $m\in\mathbb Z_{>0}$. We define 
{\small $$S^0(X, mK_X):=\bigcap_{e\geq 0} \Image\left({\rm Tr}: H^0(X, K_X+(m-1)p^eK_X) \to 
H^0(X, mK_X)\right).$$} 
where ${\rm Tr}$ is defined by 
the trace map $F^e_*\omega_X \to \omega_X$. 
For more details, see Remark~\ref{rem-S0} and \cite[Lemma 2.2.3]{HP}.}
\item{We define 
{\small $$\kappa_S(X):=\max\{r\,|\,\dim S^0(X, mK_X)=O(m^r)\,\,{\rm for\,\,sufficiently\,\,divisible}\,\, m\}.$$}
This definition is the same as the one of \cite[Subsection~4.1]{HP}.}
\end{enumerate}
\end{dfn}

\begin{rem}\label{rem-S0}
The trace map $F^e_*\omega_X \to \omega_X$ in Definition~\ref{def-S0} 
is obtained by applying the functor $\mathcal Hom_{\MO_X}(-, \omega_X)$ 
to the Frobenius $\MO_X \to F_*^e\MO_X$. 
Indeed, the Grothendieck duality implies 
$\mathcal Hom_{\MO_X}(F_*^e\MO_X, \omega_X)\simeq F_*^e\omega_X$. 
Thus, we obtain the trace map $F^e_*\omega_X \to \omega_X$. 

By the construction, if $X$ is $F$-split, then 
the trace map $F^e_*\omega_X \to \omega_X$ is a split surjection. 
Therefore, in this case, 
$H^0(X, mK_X)\neq 0$ (resp. $\kappa(X)\geq 0$) implies 
$S^0(X, mK_X)\neq 0$ (resp. $\kappa_S(X)\geq 0$). 
\end{rem}

We check $\kappa_S(X)=0$ to apply \cite[Theorem~1.1.1]{HP} in the proof of Theorem~\ref{main-theorem2}. 

\begin{lem}\label{K-torsion}
Let $X$ be a smooth projective variety. 
If $X$ is $F$-split and $K_X$ is pseudo-effective, 
then the following assertions hold. 
\begin{enumerate}
\item{$(p^e-1)K_X\sim 0$ for every $e\in\mathbb Z_{>0}$.}
\item{$\kappa_S(X)=0$.}
\end{enumerate}
\end{lem}

\begin{proof}
(1) By Lemma~\ref{Kodaira-negative}, 
we obtain $-(p^e-1)K_X\sim E$ where $E$ is an effective divisor. 
Then, the pseudo-effectiveness of $K_X$ implies that $E=0$ (cf. \cite[Lemma~5.4]{6}).

(2) By (1), we obtain $\kappa(X)=0$. 
By \cite[Lemma~4.1.3]{HP}, 
it suffices to show $\kappa_S(X)\geq 0.$ 
By Remark~\ref{rem-S0}, 
$\kappa(X)\geq 0$ implies $\kappa_S(X)\geq 0.$
\end{proof}

The following lemma is a key to show Theorem~\ref{main-theorem2}. 

\begin{lem}\label{CY-p-torsion}
Let $X$ be a smooth projective variety. 
Fix $e\in\mathbb Z_{>0}$. 
Assume the following conditions. 
\begin{itemize}
\item{$F_*^e\MO_X\simeq \bigoplus_j M_j$ where each $M_j$ is a line bundle. }
\item{$K_X$ is pseudo-effective.}
\end{itemize}
Then, the following assertions hold. 
\begin{enumerate}
\item{$M_j^{p^e}\simeq \MO_X$ for every $j$. }
\item{The number of $p^e$-torsion line bundles on $X$ is 
equal to $p^{e\cdot \dim X}$.}
\end{enumerate}
\end{lem}

\begin{proof}
(1) 
By Lemma~\ref{F-split}, 
$X$ is $F$-split. 
Thus, Lemma~\ref{K-torsion} implies $(p^{e}-1)K_X\sim 0$. 
Fix an index $j_0$ and we show $M_{j_0}^{p^e}\simeq \MO_X$. 
We can write 
$$F_*^e\MO_X=M_{j_0}\oplus(\bigoplus_{j\neq j_0} M_j).$$
Tensor $M_{j_0}^{-1}$ and we obtain 
$$H^0(X, M_{j_0}^{-p^e})\simeq H^0(X, \MO_X)\oplus\cdots.$$
In particular, we obtain $H^0(X, M_{j_0}^{-p^e})\neq 0$. 
On the other hand, by applying 
$\mathcal Hom_{\MO_X}(-, \omega_X)$, we have 
\begin{eqnarray*}
F_*^e\omega_X
&\simeq&\mathcal Hom_{\MO_X}(F_*^e\MO_X, \omega_X)\\
&\simeq&\mathcal Hom_{\MO_X}(\bigoplus_j M_j, \omega_X)\\
&\simeq&\bigoplus_j (M_j^{-1}\otimes \omega_X)\\
\end{eqnarray*}
where the first isomorphism follows from the Grothendieck duality theorem for finite morphisms. 
Tensor $\omega_X^{-1}$ and we obtain 
$$F_*^e\MO_X\simeq F_*^e(\omega_X^{1-p^e})\simeq 
(F_*^e\omega_X)\otimes_{\MO_X}\omega_X^{-1}\simeq \bigoplus_j M_j^{-1}.$$
Then, tensor $M_{j_0}$, and we obtain $H^0(X, M_{j_0}^{p^e})\neq 0$. 
Therefore, $M_{j_0}^{p^e}\simeq \MO_X$. 

(2) 
By Lemma~\ref{F-split}, $X$ is $F$-split. 
Then, the assertion follows from (1) and Lemma~\ref{p-torsion}. 
\end{proof}

Ordinary abelian varieties satisfy the condition that $F_*^e\MO_X$ 
is decomposed into line bundles. 

\begin{lem}\label{ord-abel}
Let $A$ be a $d$-dimensional ordinary abelian variety. 
Fix $e\in\mathbb Z_{>0}$. 
Let $\{M_j^{(e)}\}_{j\in J}$ be the set of the $p^e$-torsion line bundles on $X$. 
Then, the following assertions hold. 
\begin{enumerate}
\item{$F_*^e\MO_A\simeq \bigoplus_{j\in J} M_j^{(e)}$. }
\item{$M_j^{(e)}\in\Pic^0(A)$ for every $j\in J$. }
\end{enumerate}
\end{lem}

\begin{proof}
The number of $p^e$-torsion line bundles in $\Pic^0(X)$ is $p^{ed}$. 
Apply Lemma~\ref{p-torsion} and we obtain the assertion.
\end{proof}

We also need the following lemma.

\begin{lem}\label{determinant}
Let $X$ be a proper normal variety. 
Fix $e\in\mathbb Z_{>0}$. 
Assume that there are mutually distinct $p^e$-torsion line bundles 
$L_1, \cdots, L_{p^{e\dim X}}$ on $X$. 
Let $F_*^e\MO_X \simeq E\oplus E'$ where 
$E\neq 0$ is an indecomposable coherent sheaf and $E'$ is a coherent sheaf. 
Then, the following assertions hold. 
\begin{enumerate}
\item{If $\rank\,E<p$, then $F_*^e\MO_X \simeq \bigoplus_{i=1}^{p^{e\dim X}} L_i.$}
\item{If $\rank\,E=p$, then $E\otimes_{\MO_X} L_i \simeq E\otimes_{\MO_X} L_j$ 
for some $1\leq i< j \leq p^{e\dim X}$.}
\end{enumerate}
\end{lem} 

\begin{proof}
Set $X_{{\rm reg}} \subset X$ to be the regular locus of $X$. 
Since $(F_*^e\MO_X)|_{X_{{\rm reg}}}$ is locally free, 
$E|_{X_{{\rm reg}}}$ is also locally free. 

We show that $E$ is reflexive. 
Let 
$$F_*^e\MO_X \simeq E_1 \oplus \cdots \oplus E_s$$
be a decomposition into indecomposable sheaves with $E_1 \simeq E$. 
Take the double dual. 
Since $F_*^e\MO_X$ is reflexive, each $E_i$ is reflexive by the Krull--Schmidt theorem (\cite[Theorem 2]{Atiyah}).

\medskip
(1) 
We show that 
$$E\otimes_{\MO_X} L_i\not\simeq E\otimes_{\MO_X} L_j$$
for every $1\leq i< j\leq p^{e\dim X}$. 
Assume $E\otimes_{\MO_X} L_i\simeq E\otimes_{\MO_X} L_j$ for some $1\leq i< j\leq p^{e\dim X}$. 
Then, we obtain 
$$\det\,(E|_{X_{{\rm reg}}})\otimes_{\MO_{X_{{\rm reg}}}} (L_i|_{X_{{\rm reg}}})^{\rank\,E} \simeq 
\det\,(E|_{X_{{\rm reg}}})\otimes_{\MO_{X_{{\rm reg}}}} (L_j|_{X_{{\rm reg}}})^{\rank\,E}.$$
By $1\leq \rank\,E<p$, we obtain $L_i \simeq L_j$, which is a contradiction. 

Thus $E\otimes_{\MO_X} L_i$ is 
also an indecomposable direct summand of $F^e_*\MO_X$. 
Therefore, we see $\rank\,E=1$ and 
$$F^e_*\MO_X \simeq \bigoplus_{i=1}^{p^{e\dim X}}E\otimes_{\MO_X} L_i.$$ 
Since $E$ is a divisorial sheaf, $X$ is $F$-split by Lemma~\ref{F-split}. 
Then, the assertion follows from Lemma~\ref{p-torsion}. 

(2) 
Assume that 
$E\otimes_{\MO_X} L_i \not\simeq E\otimes_{\MO_X} L_j$ 
for every $1\leq i< j \leq p^{e\dim X}$. 
Let us derive a contradiction. 
Since $E$ is indecomposable, so is $E\otimes_{\MO_X} L_i$ for every $i$. 
Moreover, $E\otimes_{\MO_X} L_i$ is also a direct summand of $F_*^e\MO_X$. 
Thus, by the Krull--Schmidt theorem (\cite[Theorem 2]{Atiyah}), 
we obtain 
$$F_*^e\MO_X \simeq \bigoplus_{i=1}^{p^{e\dim X}}E\otimes_{\MO_X} L_i\oplus \cdots.$$
Then, we obtain the following contradiction: 
$$p^{e \dim X} =\rank(F_*^e\MO_X) \geq p^{e\dim X}\times \rank\,E=p^{e\dim X}\times p.$$
\end{proof}






We show the main theorem of this paper. 

\begin{thm}\label{main-theorem2}
Let $X$ be a smooth projective variety. 
Assume that the following conditions hold. 
\begin{itemize}
\item{For infinitely many $e\in\mathbb Z_{>0}$, 
$F_*^e\MO_X\simeq \bigoplus_j M_j^{(e)}$ where each $M_j^{(e)}$ is a line bundle.}
\item{$K_X$ is pseudo-effective.}
\end{itemize}
Then, $X$ is an ordinary abelian variety. 
\end{thm}

\begin{proof}
Let 
$$\alpha:X \to A:=\Alb(X)$$
be the Albanese morphism. 

\begin{step}\label{main-theorem1}
In this step, we show the following assertions. 
\begin{enumerate}
\item{The Albanese morphism $\alpha:X\to A$ is surjective. }
\item{The Albanese variety $A$ is an oridnary abelian variety such that $\dim X=\dim A$. }
\item{For every $e\in\mathbb Z_{>0}$, 
$F_*^e\MO_X\simeq \bigoplus_j M_j^{(e)}$ where each $M_j^{(e)}$ is a $p^e$-torsion line bundle. }
\end{enumerate}
\end{step}
\begin{proof}[Proof of Step~\ref{main-theorem1}]
(1) 
Lemma~\ref{F-split} implies that $X$ is $F$-split. 
By Lemma~\ref{K-torsion}, we see $\kappa_S(X)=0$. 
Thus we can apply \cite[Theorem~1.1.1(1)]{HP}. 
Then, the Albanese morphism $\alpha:X\to \Alb(X)$ is surjective.

(2) 
By (1), we obtain $\dim \Pic^0(X)_{\red}\leq \dim X$. 
Set $r_X$ to be the $p$-rank of $\Pic^0(X)_{\red}$. 
It suffices to show that $r_X=\dim X$. 
By Lemma~\ref{CY-p-torsion} and an assumption, 
the number of $p^e$-torsion line bundles is equal to $p^{e\dim X}$ for infinitely many $e\in\mathbb Z_{>0}$. 
By Proposition~\ref{picard-exact}(2), 
we can find an integer $\xi>0$ such that 
$$p^{er_X} \leq p^{e\dim X}=|\Pic(X)[p^e]| \leq  p^{er_X}\times \xi,$$
for infinitely many $e>0$. 
Taking the limit $e\to \infty$, we obtain $r_X=\dim X.$ 

(3) 
The assertion follows from (2) and Lemma~\ref{p-torsion}. 
This completes the proof of Step~\ref{main-theorem1}.
\end{proof}

By Step~\ref{main-theorem1}, 
the Albanese morphism $\alpha:X \to A$ is a generically finite surjective morphism and 
$A$ is an ordinary abelian variety. 
We obtain the following decomposition 
$$\alpha:X \overset{f}\to Y \overset{g}\to Z \overset{h}\to A$$
such that 
\begin{itemize}
\item{$Y$ and $Z$ are projective normal varieties.}
\item{$f$ is a birational morphism, and $g$ and $h$ are finite surjective morphisms.}
\item{$g$ is purely inseparable and $h$ is separable.}
\end{itemize}
Note that we can find such a decomposition as follows. 
First, we take the Stein factorization of $\alpha$ and we obtain $Y$. 
Then $f:X \to Y$ is birational and $Y \to A$ is finite. 
Second, take the separable closure $L$ of $K(A)$ in $K(X)=K(Y)$ and 
consider the normalization $Z$ of $A$ in $L$. 

\begin{step}\label{Ysmooth}
$Y$ is smooth. 
\end{step}
\begin{proof}[Proof of Step~\ref{Ysmooth}]
Since $f_*\MO_X=\MO_Y$, $Y$ is $F$-split. 
By Lemma~\ref{ord-abel}, there are the mutually distinct 
$p$-torsion line bundles $M_1, \cdots, M_{p^{\dim X}}$ on $A$ such that $M_i \in\Pic^0(A)$. 
By Subsection~\ref{Albanese}, 
$\alpha^*M_1, \cdots, \alpha^*M_{p^{\dim X}}$ are mutually distinct $p$-torsion line bundles on $X$. 
Thus, the number of $p$-torsion line bundles on $Y$ is at least $p^{\dim X}=p^{\dim Y}$. 
Then, by Lemma~\ref{p-torsion}, 
$F_*\MO_Y\simeq \bigoplus_{j \in J} L_j$ for some $p$-torsion line bundles $L_j$ on $Y$. 
Therefore $Y$ is smooth by Kunz's criterion. 
\end{proof}

\begin{step}\label{XY}
$f$ is an isomorphism. 
\end{step}
\begin{proof}[Proof of Step~\ref{XY}]
We can write 
$$K_X=f^*K_Y+E$$
where $E$ is an $f$-exceptional divisor. 
Since $Y$ is smooth and hence terminal (cf. \cite[Section~2.3]{KM}), $E$ is effective. 
Since $K_X\equiv 0$, we see that $E$ is $f$-nef. 
By the negativity lemma (cf. \cite[Lemma~3.39]{KM}), we see $E=0$. 
Therefore, $K_X=f^*K_Y$. 
Thus, the codimension of $\Ex(f)$ in $X$ is at least two. 
Since $Y$ is smooth, $f$ is an isomorphism. 
\end{proof}

Now, we have 
$$\alpha:X \overset{g}\to Z \overset{h}\to A$$
such that 
\begin{itemize}
\item{$Z$ is projective normal variety.}
\item{$g$ is a finite surjective purely inseparable morphism. }
\item{$h$ is a finite surjective separable morphism. }
\end{itemize}

\begin{step}\label{separable}
If $g$ is an isomorphism, then $\alpha$ is also an isomorphism. 
\end{step}
\begin{proof}[Proof of Step~\ref{separable}]
We see that the albanese morphism 
$$\alpha=h:X \to A$$
is a finite surjective separable morphism. 
Since $K_X$ is numerically trivial and $K_A \sim 0$, 
$\alpha:X \to A$ is etale in codimension one. 
Then, by the Zariski--Nagata purity, $\alpha$ is etale. 
By \cite[Section 18, Theorem]{Mumford}, 
$X$ is also an ordinary abelian variety. 
This completes the proof of Step~\ref{separable}. 
\end{proof}

\begin{step}\label{p-insep}
$g$ is an isomorphism. 
\end{step}

\begin{proof}[Proof of Step~\ref{p-insep}]
Assume that $g$ is not an isomorphism. 
Then, we can find 
$$\alpha:X \overset{\varphi}\to W \to Z \to A,\,\,\,\, \beta:W \to A$$
which satisfies the following properties. 
\begin{itemize}
\item{$W$ is a projective normal variety.}
\item{$\varphi:X \to W$ and $W \to Z$ are finite surjective purely inseparable morphisms 
with $[K(X):K(W)]=p$.}
\end{itemize}
Since $A$ is an ordinary abelian variety, 
there are mutually distinct $p$-torsion line bundles $M_1, \cdots, M_{p^{\dim X}}$ 
on $A$ which form a subgroup of $\Pic^0 A$ (Lemma~\ref{ord-abel}). 

\medskip

\begin{claim}
We prove the following assertions. 
\begin{enumerate}
\item[(a)]{$F_*\MO_{W}\simeq \varphi_*\MO_{X}\oplus E$ for some coherent sheaf $E$.}
\item[(b)]{$F_*\MO_{W}\simeq \beta^*M_1\oplus \cdots\oplus \beta^*M_{p^{\dim X}}.$}
\end{enumerate}
\end{claim}

\medskip 

\begin{claimproof}
(a) 
Since $[K(X):K(W)]=p$, 
the Frobenius map $F_{W}$ factors through $\varphi$: 
$$F_{W}:W \overset{\mu}\to X \overset{\varphi}\to W.$$
Since $\mu$ is a finite purely inseparable morphism, 
there is $e\in\mathbb Z_{>0}$ such that 
$F^{e}_X$ factors through $\mu$: 
$$F^{e}_X:X \to W \overset{\mu}\to X.$$
Since $X$ is $F$-split, the identity homomorphism ${\rm id}_{\MO_X}$ 
factors through $\mu_*\MO_{W}$: 
$${\rm id}_{\MO_X}:\MO_X \to \mu_*\MO_{W} \to (F^{e}_X)_*\MO_X \to \MO_X.$$
Thus, we see 
$$\mu_*\MO_{W}\simeq \MO_{X}\oplus E_1$$
for some coherent sheaf $E_1$ on $X$. 
Take the push-forward by $\varphi$ and we obtain 
$$(F_{W})_*\MO_{W}\simeq \varphi_*\MO_X \oplus \varphi_*E_1.$$

(b) 
Set $L_i:=\beta^*M_i$. 
By Subsection~\ref{Albanese}, 
$L_1, \cdots, L_{p^{\dim X}}$ are mutually distinct $p$-torsion line bundles on $W$ 
such that $\{L_1, \cdots, L_{p^{\dim X}}\}$ forms a subgroup  of $\Pic\,W$ and that 
$$\varphi^*L_i \not\simeq \varphi^*L_j$$ for every $1\leq i< j\leq p^{\dim X}$. 
There are the following two cases: 
\begin{itemize}
\item{$\varphi_*\MO_{X}$ is not indecomposable.}
\item{$\varphi_*\MO_{X}$ is indecomposable.}
\end{itemize}

Assume that $\varphi_*\MO_X$ is not indecomposable. 
Then, $F_*\MO_W$ has an indecomposable direct summand of rank $<p$. 
Therefore, by Lemma~\ref{determinant}(1), 
we obtain 
$$F_*\MO_{W}\simeq \beta^*M_1\oplus \cdots\oplus \beta^*M_{p^{\dim X}}.$$
This is what we want to show. 

Assume that $\varphi_*\MO_{X}$ is indecomposable. 
Since $\rank(\varphi_*\MO_{X})=p$, 
we can apply Lemma~\ref{determinant}(2) and can find 
$$\varphi_*\MO_{X}\otimes L_i \simeq \varphi_*\MO_{X}\otimes L_j$$
for some $1\leq i< j \leq p^{\dim X}$. 
Since $\{L_1, \cdots, L_{p^{\dim X}}\}$ is a group, 
we obtain $L_i^{-1}\otimes_{\MO_X} L_j \simeq L_r$ for some $1\leq r \leq p^{\dim X}$ 
with $\varphi^*L_r\not\simeq \MO_X$. 
Tensor $L_i^{-1}$ and we see 
$$\varphi_*\MO_X \simeq \varphi_*\MO_X\otimes L_r \simeq \varphi_*(\varphi^*L_r).$$
Then, taking $H^0$, we obtain the following contradiction 
$$0\neq H^0(X, \MO_X) \simeq H^0(X, \varphi^*L_r)=0,$$
where the last equality holds because $\varphi^*L_r$ 
is a non-trivial $p$-torsion line bundle. 
This completes the proof of Claim. 
\end{claimproof}

\medskip

By the Krull--Schmidt theorem (\cite[Theorem 2]{Atiyah}), 
the assertions (a) and (b) in Claim imply 
$$\varphi_*\MO_X=\bigoplus_{j\in J} \beta^*M_j$$
for some $J \subset \{1, \cdots, p^{\dim X}\}$. 
Since $\#J=p$, we obtain $M_{j_0} \not\simeq \MO_A$ for some $j_0\in J$. 
By Subsection~\ref{Albanese}, 
we see that $\alpha^*M_{j_0}\not\simeq \MO_X$. 
Since $\alpha^*M_{j_0}$ is a non-trivial $p$-torsion line bundle, we obtain 
$$H^0(X, \alpha^*M_{j_0}^{-1})=0.$$
On the other hand, 
we obtain 
$$\varphi_*\alpha^*M_{j_0}^{-1}
\simeq \varphi_*\varphi^*\beta^*M_{j_0}^{-1}
\simeq \varphi_*\MO_X \otimes \beta^*M_{j_0}^{-1} $$
$$\simeq (\bigoplus_{j\in J} \beta^*M_j) \otimes \beta^*M_{j_0}^{-1} 
\simeq \MO_W\oplus(\bigoplus_{j\neq j_0} \beta^*M_j\otimes \beta^*M_{j_0}^{-1}),$$
which implies 
$$H^0(X, \alpha^*M_{j_0}^{-1})\neq 0.$$
This is a contradiction. 
Thus, $g:X \to Z$ is an isomorphism. 
This completes the proof of Step~\ref{p-insep}. 
\end{proof}

Step~\ref{separable} and Step~\ref{p-insep} 
imply the assertion in the theorem. 
\end{proof}


\section{On the behavior of $F^e_*\MO_X$ for some special varieties}\label{curve}

In the former sections, we investigate varieties $X$ such that 
$F_*\MO_X$ is decomposed into line bundles. 
In this section, we study the behavior of $F_*\MO_X$ for some special varieties.

\subsection{Abelian varieties}

In this subsection, we show Theorem~\ref{abel-main}. 
We recall some results essentially obtained by \cite{Oda}.



\begin{thm}[Oda]\label{Oda1}
Let $f:X \to Y$ be an isogeny of abelian varieties over $k$. 
Set $\hat{f}:\hat{Y} \to \hat{X}$ to be the dual of $f$. 
Let $L\in \Pic^0(X)$. 
Then, 
$$f_*L \simeq {\rm pr}_{1*}(\mathcal P_Y|_{Y\times{\hat{f}^{-1}([L])}})$$
where $\mathcal P_Y$ is the normalized Poincare line bundle of $(Y, 0)$ 
and ${\rm pr}_1$ is the first projection. 
\end{thm}

\begin{proof}
We can apply the same argument as \cite[Corollary~1.7]{Oda}. 
\end{proof}

\begin{thm}[Oda]\label{Oda2}
Let $X$ be an abelian variety. 
Let $S\subset \hat X$ be a closed subscheme of the dual abelian variety $\hat X$. 
If $S$ is zero-dimensional and Gorenstein, then the following assertions hold. 
\begin{enumerate}
\item{There exists an isomorphism between non-commutative $k$-algebras: 
$${\rm End}_{\MO_X}({\rm pr}_{1*}(\mathcal P_X|_{X\times S})) \simeq \Gamma(S, \MO_S).$$
In particular, ${\rm End}_{\MO_X}({\rm pr}_{1*}(\mathcal P_X|_{X\times S}))$ is a commutative ring.}
\item{If $S$ is one point, that is, $\Gamma(S, \MO_S)$ is a local ring, 
then ${\rm pr}_{1*}(\mathcal P_X|_{X\times S})$ is an indecomposable sheaf.}
\end{enumerate}
\end{thm}

\begin{proof}
(1) holds from \cite[Corollary 1.12]{Oda}. 
We show (2). 
Assuming ${\rm pr}_{1*}(\mathcal P_X|_{X\times S})\simeq E_1 \oplus E_2$ with $E_i\neq 0$, 
we derive a contradiction. 
By (1), the ring ${\rm End}_{\MO_X}({\rm pr}_{1*}(\mathcal P_X|_{X\times S}))$ is a commutative ring. 
We obtain idempotents 
${\rm id}_{E_1}\times 0_{E_2}$ and $0_{E_1}\times {\rm id}_{E_2}$ such that 
${\rm id}_{E_1}\times 0_{E_2}+0_{E_1}\times {\rm id}_{E_2}$ is the unity of 
the ring ${\rm End}_{\MO_X}({\rm pr}_{1*}(\mathcal P_X|_{X\times S}))$. 
Therefore, we obtain 
$$\Gamma(S, \MO_S) \simeq {\rm End}_{\MO_X}({\rm pr}_{1*}(\mathcal P_X|_{X\times S})) \simeq A\times B$$ 
for some non-zero rings $A$ and $B$. 
But, $\Gamma(S, \MO_S)$ is a local ring. 
This is a contradiction. 
\end{proof}

We show the main theorem of this subsection. 

\begin{thm}\label{abel-main}
Let $X$ be an abelian variety. 
Set $r_X$ to be the $p$-rank of $X$. 
Let $L\in\Pic^0(X)$. 
Then, for every $e\in\mathbb Z_{>0}$, we obtain 
$$F_*^eL \simeq E_1 \oplus \cdots \oplus E_{p^{er_X}}$$
where each $E_i$ is an indecomposable locally free sheaf of rank $p^{e(\dim X-r_X)}$. 
\end{thm}

\begin{proof}
Fix $e\in\mathbb Z_{>0}$. 
Consider the absolute Frobenius morphism $F^e_X:X \to X.$ 
Set $X^{(p^e)}:=X \times_{k, F_k^e} k$ and we obtain 
$$F^e_X:X \xrightarrow{F^{e, {\rm rel}}_X} X^{(p^e)} \overset{\beta}\to X.$$
where $\beta$ is a non-$k$-linear isomorphism of schemes and 
$$F^{e, {\rm rel}}_X:X \to X^{(p^e)}$$ 
is $k$-linear. 
Thus, it suffices to show that 
$$(F^{e, {\rm rel}}_X)_*L \simeq E'_1 \oplus \cdots \oplus E'_{p^{er_X}}$$
for some indecomposable locally free sheaves $E'_i$ of rank $p^{e(\dim X-r_X)}$. 
Take the dual of $F^{e, {\rm rel}}_X$: 
$$\widehat{(F^{e, {\rm rel}}_X)}:\widehat{X^{(p^e)}} \to \hat X.$$

We show that the number of the fiber of every closed point of $\widehat{(F^{e, {\rm rel}}_X)}$ is $p^{er_X}$. 
Since $\widehat{(F^{e, {\rm rel}}_X)}(k)$ is a group homomorphism, 
the numbers of all the fibers are the same. 
Thus, it suffices to prove that the number of 
$\widehat{(F^{e, {\rm rel}}_X)}^{-1}(0_{\hat X})=\Ker(\widehat{(F^{e, {\rm rel}}_X)}(k))$ is $p^{er_X}$. 
This is equivalent to show that the number of line bundles $M\in \Pic^0(X^{(p^e)})=\widehat{X^{(p^e)}}(k)$ 
such that $(F^{e, {\rm rel}}_X)^*M\simeq \MO_X$ is $p^{er_X}$. 
Since $\beta:X^{(p^e)}\to X$ is an isomorphism, 
we prove that the number of line bundles $N\in \Pic^0(X)$ 
such that $N^{p^e}=(F^{e}_X)^*N\simeq \MO_X$ is $p^{er_X}$. 
This follows from the definition of the $p$-rank.

Taking the separable closure, we obtain 
$$\widehat{(F^{e, {\rm rel}}_X)}:\widehat{X^{(p^e)}} \overset{g}\to Y \overset{h}\to \hat X,$$
where $Y$ is a normal projective variety, 
$g$ is a finite surjective purely inseparable morphism and 
$h$ is a finite surjective separable morphism. 
Since the numbers of every fiber of $\widehat{(F^{e, {\rm rel}}_X)}$ are the same, 
$h$ is an etale morphism. 
In particular, $Y$ is an abelian variety (\cite[Section 18, Theorem]{Mumford}) and we may assume that 
$g$ and $h$ are isogenies. 
Take the duals again and we obtain 
$$F^{e, {\rm rel}}_X:X \xrightarrow{\hat h} \hat Y \xrightarrow{\hat g} X^{(p^e)}.$$

Let 
$$\Ker(\widehat{F^{e, {\rm rel}}_X})=g^{-1}([M_1]) \amalg \cdots \amalg g^{-1}([M_{p^{er_X}}])$$
be the decomposition into one point schemes. 
By Theorem~\ref{Oda1}, we obtain 
\begin{eqnarray*}
(F^{e, {\rm rel}}_X)_*\MO_X 
&\simeq& {\rm pr}_{1*}(\mathcal P_{X^{(p^e)}}|_{X^{(p^e)}\times \Ker(\widehat{F^{e, {\rm rel}}_X})})\\
&\simeq& {\rm pr}_{1*}(\mathcal P_{X^{(p^e)}}|_{X^{(p^e)}\times g^{-1}([M_1])})\oplus\cdots\oplus 
{\rm pr}_{1*}(\mathcal P_{X^{(p^e)}}|_{X^{(p^e)}\times g^{-1}([M_{p^{er_X}}])})\\
&\simeq&\hat g_*M_1\oplus \cdots \oplus \hat g_*M_{p^{er_X}}.
\end{eqnarray*}
Thus, it suffices to show that each locally free sheaf 
$${\rm pr}_{1*}(\mathcal P_{X^{(p^e)}}|_{X^{(p^e)}\times g^{-1}([M_j])}) \simeq \hat g_*M_j$$ 
is indecomposable. 
We see that $g^{-1}([M_j])$ is one point. 
Thus, if $g^{-1}([M_j])$ is Gorenstein, 
then ${\rm pr}_{1*}(\mathcal P_{X^{(p^e)}}|_{X^{(p^e)}\times g^{-1}([M_j])})$ is indecomposable by Theorem~\ref{Oda2}(2). 
Since $g$ is finite and $Y$ is smooth, 
$g^{-1}([M_j])$ is a local complete intersection scheme. 
In particular, $g^{-1}([M_j])$ is Gorenstein. 
\end{proof}

\subsection{Curves}

In this subsection, we show Theorem~\ref{main-curve}. 
We need the following result from the theory of stable vector bundles. 

\begin{thm}\label{higher-genus}
Let $X$ be a smooth projective curve of genus $g\geq 2$. 
Let $L$ be a line bundle on $X$. 
Then, $F_*^eL$ is indecomposable for every $e\in\mathbb Z_{>0}$. 
\end{thm}

\begin{proof}
Since $L$ is a line bundle, $L$ is a stable vector bundle. 
Then, by \cite[Theorem~2.2]{Sun}, $F_*^eL$ is also a stable vector bundle. 
Since stable vector bundles are indecomposable, $F_*^eL$ is indecomposable. 
\end{proof}

We show the main theorem of this subsection. 

\begin{thm}\label{main-curve}
Let $X$ be a smooth projective curve of genus $g$. 
Fix an arbitrary positive integer $e$. 
Then the following assertions hold. 
\begin{enumerate}
\item[$(0)$]{If $g=0$, then $F^e_*\MO_X \simeq \bigoplus L_j$ where 
every $L_j$ is a line bundle.}
\item[$(1{\rm or})$]{If $g=1$ and $X$ is an ordinary elliptic curve, 
then $F^e_*\MO_X \simeq \bigoplus L_j$ where every $L_j$ is a line bundle.}
\item[$(1{ \rm ss})$]{If $g=1$ and $X$ is a supersingular elliptic curve, 
then $F^e_*\MO_X$ is indecomposable.}
\item[$(2)$]{If $g\geq 2$, then $F^e_*\MO_X$ is indecomposable.}
\end{enumerate}
\end{thm}

\begin{proof}
The assertion (0) immediately follows from the fact that 
every locally free sheaf of finite rank on $\mathbb P^1$ is decomposed into the direct sum of line bundles. 
The assertions $(1{\rm or})$ and $(1{ \rm ss})$ hold by Theorem~\ref{abel-main}. 
The assertion (2) follows from Theorem~\ref{higher-genus}.
\end{proof}

By Theorem~\ref{main-curve}, it is natural to ask the following question. 

\begin{ques}\label{q-general-type}
If $X$ is a smooth projective surface $X$ of general type, 
then is $F_*\MO_X$ indecomposable?
\end{ques}

As far as the authors know, this question is open. 
On the other hand, 
if we drop the assumption that $X$ is smooth, 
then there exists a counter-example as follows. 
For a related result, see also \cite[Example~3.5]{Hirokado}.

\begin{thm}
There exists a projective normal surface $X$ which satisfies the following properties. 
\begin{enumerate}
\item{The singularities of $X$ are at worst canonical.}
\item{$K_X$ is ample.}
\item{$F_*\MO_X$ is not indecomposable.}
\end{enumerate}
\end{thm}

\begin{proof}
Let $S$ be an ordinary abelian surface. 
Fix a very ample line bundle $H$ on $S$. 
Let $s\in H^0(X, H^p)$ be a general element and 
set 
$$\pi:X:=\Spec_S(\MO_S\oplus H^{-1} \oplus\cdots\oplus H^{-(p-1)}) \to S$$
to be the finite purely inseparable morphism where 
the $\MO_S$-algebra $\MO_S\oplus H^{-1} \oplus\cdots\oplus H^{-(p-1)}$ 
is defined by $s\in H^0(X, H^p)$. 
By \cite[Remark~3.5(1)]{Liedtke}, we can apply \cite[Theorem~3.4]{Liedtke} for $\mathcal L:=H$. 
Since the scheme $X$ constructed above 
is the same as the $\alpha_{\mathcal L}$-torsor $\delta(s)$ appearing in \cite[Theorem~3.4]{Liedtke}. 
Therefore, $X$ is normal and has at worst $A_{p-1}$-singularities. 
Thus (1) holds. 
We see 
$$K_X=\pi^*K_S+(p-1)\pi^*H \sim (p-1)\pi^*H,$$
which implies (2). 

We show (3). 
Since $\pi:X \to S$ is a finite purely inseparable morphism of degree $p$, 
the absolute Frobenius morphisms of $X$ and $S$ factors through $\pi$: 
$$F_S:S \to X \overset{\pi}\to S, \,\,\, F_X:X \overset{\pi}\to S \overset{\varphi}\to X.$$
Since $S$ is $F$-split, the identity homomorphism ${\rm id}_{\MO_S}$ 
factors through $\pi_*\MO_X$: 
$${\rm id}_{\MO_S}:\MO_S \to \pi_*\MO_X \to (F_S)_*\MO_S \to \MO_S.$$
This implies 
$$\pi_*\MO_X\simeq \MO_S \oplus E$$
for some coherent sheaf $E$. 
Taking the push-forward by $\varphi$, we see 
$$(F_X)_*\MO_X=\varphi_*\pi_*\MO_X\simeq \varphi_*\MO_S \oplus \varphi_*E.$$
This implies (3). 
\end{proof}

\begin{rem}
If $X$ is a smooth projective curve of general type, 
then $F_*\MO_X$ is indecomposable by Theorem~\ref{higher-genus}. 
Theorem~\ref{higher-genus} depends 
on the theory of the stable vector bundles. 
For the $2$-dimensional case, 
a similar result is obtained by Kitadai--Sumihiro (\cite{KS}), Liu--Zhou (\cite{LZ}), 
and Sun (\cite{Sun2}). 
For example, \cite[Theorem~4.9 and Remark~4.10]{Sun2} imply that $F_*\MO_X$ is indecomposable 
under the assumptions that $\mu(\Omega_X^1)>0$ and $\Omega_X^1$ is semi-stable. 
\end{rem}


\begin{thebibliography}{Thomsen}


\bibitem[Achinger]{Achinger}
{P. Achinger},
{\em A characterization of toric varieties in characteristic $p$,} 
{to appear in Int. Math. Res. Notices}.



\bibitem[Atiyah1]{Atiyah}
{M. F. Atiyah},
{\em On the Krull-Schmidt theorem with application to sheaves, }
{Bull. Soc. Math. France, tome {\textbf{84}} (1956), 307--317}.


\bibitem[Atiyah2]{Atiyah2}
{M. F. Atiyah},
{\em Vector bundles over an elliptic curve, }
{Proc. London Math. Soc, {\textbf{7}} (1957), 414--452.}






\bibitem[FGAex]{FGAex}
{B. Fontechi, L. G\"{o}ttsche, L. Illusie, S. L. Kleiman, N. Nithure, A. Vistoli}, 
{\em Fundamental Algebraic Geometry: Grothendieck's FGA Explained}, 
{Math. Surveys and Monographs, Vol. {\textbf{123}} (2005)}. 

\bibitem[6]{6}
{Y. Gongyo, Z. Li, Z. Patakfalvi, K. Schwede, H. Tanaka, R. Zong}, 
{\em On rational connectedness of globally F-regular threefolds}, 
{Adv. Math. {\textbf{280}} (2015), 47--78.} 

\bibitem[HP]{HP}
{C. D. Hacon, Z. Patakfalvi},
{\em Generic vanishing in characteristic $p>0$ and the characterization of ordinary abelian varieties}, 
{to appear in Amer. J. of Math. } 

\bibitem[Hirokado]{Hirokado}
{M. Hirokado},
{\em Zariski surfaces as quotients of Hirzebruch surfaces by 1-foliations}, 
{Yokohama Math. J. vol. 47 (2000), 103--120.} 




\bibitem[KS]{KS}
{Y. Kitadai, H Sumihiro}, 
{\em Canonical filtrations and stability of direct images by Frobenius morphisms}, 
{Tohoku Math. J. {\textbf{60}}, 287--301 (2008)}. 


\bibitem[KM]{KM}
{J.~Koll\'ar,~S.~Mori},
{\em Birational geometry of algebraic varieties},
{Cambrigde Tracts in Mathematics, Vol. {\textbf{134}}, 1998}.





\bibitem[Liedtke]{Liedtke}
{C.~Liedtke},
{\em The Canonical Map and Horikawa Surfaces in Positive Characteristic}, 
{Int. Math. Res. Notices (2013) Vol. {\textbf{2013}}(2), 422--462.}




\bibitem[LZ]{LZ}
{C. Liu and M. Zhou},
{\em Stability of Frobenius direct images over surfaces}, 
{Math. Z. (2015) {\textbf{280}}, 841--850}

\bibitem[MeSr]{MeSr}
{V. B. Mehta, V. Srinivas},
{\em Varieties in positive characteristic with trivial tangent bundle}, 
{Compositio Math., tome {\textbf{64}} no. 2, 191--212 (1987)}.


\bibitem[Mumford]{Mumford}
{D. Mumford},
{\em Abelian varieties},
{Tata Institute of Fundamental Research Studies in Mathematics, No. {\textbf{5}}, (1970)}.


\bibitem[MuSr]{MuSr}
{M.~Musta\c{t}\u{a}, V. Srinivas},
{\em Ordinary varieties and the comparison between multiplier ideals and test ideals}, 
{Nagoya Math. J. Volume {\textbf{204}} (2011), 125--157}




\bibitem[Oda]{Oda}
{T. Oda},
{\em Vector bundles on an elliptic curve}, 
{Nagoya Math. J., Vol. {\textbf{43}}, 41--72 (1971)}.


\bibitem[SS]{SS}
{K. Schwede, K. E. Smith},
{\em Globally $F$-regular and log Fano varieties,} 
{Adv. in Math., Vol. {\textbf{224}} Issue 3, 863--894, (2010).}


\bibitem[Sun1]{Sun}
{X. Sun},
{\em Direct images of bundles under Frobenius morphism},
{J. Algebra {\textbf{226}} (2000), no. 2, 865--874}


\bibitem[Sun2]{Sun2}
{X. Sun},
{\em Frobenius morphism and semi-stable bundles},
{Advanced Studies in Pure Mathematics {\textbf{60}} (2010), 
Algebraic Geometry in East Asia-Seoul (2008),161--182}


\bibitem[Thomsen]{Th}
{J. F. Thomsen},
{\em Frobenius direct images of line bundles on toric varieties},
{J. Algebra {\textbf{226}} (2000), no. 2, 865--874}



\end{thebibliography}
\end{document}